\documentclass[12pt]{amsart}
\usepackage[utf8]{inputenc}
\usepackage{amsmath}
\usepackage{amsfonts}
\usepackage{amssymb}
\usepackage{amsthm}
\usepackage{bbm}
\usepackage{hyperref}

\usepackage{color}

\theoremstyle{plain}
\newtheorem{theorem}{Theorem}
\newtheorem{lemma}[theorem]{Lemma}

\newtheorem{proposition}[theorem]{Proposition}
\theoremstyle{definition}

\theoremstyle{remark}

\numberwithin{equation}{section}
\numberwithin{theorem}{section}

\author{Kyle Pratt}
\address{Department of Mathematics, University of Illinois, 1409 West Green Street, Urbana, IL 61801, United States}
\email{kpratt4@illinois.edu}


\title{Average non-vanishing of Dirichlet $L$-functions at the central point}

\begin{document}
\date{}

\begin{abstract}
The Generalized Riemann Hypothesis implies that at least 50\% of the central values $L \left( \frac{1}{2},\chi\right)$ are non-vanishing as $\chi$ ranges over primitive characters modulo $q$. We show that one may unconditionally go beyond GRH, in the sense that if one averages over primitive characters modulo $q$ and averages $q$ over an interval, then at least 50.073\% of the central values are non-vanishing. The proof utilizes the mollification method with a three-piece mollifier, and relies on estimates for sums of Kloosterman sums due to Deshouillers and Iwaniec.

Note: The author has been made aware of an error in this work. It seems the error can be fixed, by using a different argument, and the author will present a correction in due course.
\end{abstract}

\maketitle

\section*{Notice}

Kaisa Matom{\"a}ki and Martin {\v C}ech have drawn my attention to an error in my paper ``Average non-vanishing of Dirichlet $L$-functions at the central point.'' There is a mistake in the calculation of the main term in Lemma 3.4; a factor of $\theta_2/\theta_3$ is missing in the calculation of the $N_2$ term. The correct expression is
\begin{align*}
N_2 =-(1+o(1))\varphi^+(q) \theta_2\int_0^1 P_2 (1-\alpha) P_3 \left(1 - \alpha \frac{\theta_2}{\theta_3} \right) d\alpha.
\end{align*}
When this error is corrected, the combination of the three mollifiers does not yield a nonvanishing proportion greater than one-half. Thus, the statement of Theorem 1.1 is, at present, unproven.

It appears that, with an alternate method, one can recover a nonvanishing proportion greater than $50\%$ (though likely not a percentage as large as that claimed in Theorem 1.1). The author is working on a fix, and will present a corrected argument in due course.

\section{Introduction}

It is widely believed that no primitive Dirichlet $L$-function $L \left( s,\chi\right)$ vanishes at the central point $s = \frac{1}{2}$. Most of the progress towards this conjecture has been made by working with various families of Dirichlet $L$-functions. Balasubramanian and Murty \cite{BalMur} showed that, in the family of primitive characters modulo $q$, a positive proportion of the $L$-functions do not vanish at the central point. Iwaniec and Sarnak \cite{IwaSar} later improved this lower bound, showing that at least $\frac{1}{3}$ of the $L$-functions in this family do not vanish at the central point. Bui \cite{Bui} improved this further to 34.11\%, and Khan and Ng\^o \cite{KhaNgo} showed at least $\frac{3}{8}$ of the central values are non-vanishing\footnote{A clear preference for ``non-vanishing'' or ``nonvanishing'' has not yet materialized in the literature. We exclusively use the former term throughout this work.} for prime moduli. Soundararajan \cite{Sou} worked with a family of quadratic Dirichlet characters, and showed that $\frac{7}{8}$ of the family do not vanish at $s = \frac{1}{2}$. These proofs all proceed through the mollification method, which we discuss in section \ref{sec: mollify} below.

If one assumes the Generalized Riemann Hypothesis, one can show that at least half of the primitive characters $\chi \pmod{q}$ satisfy $L \left( \frac{1}{2},\chi\right) \neq 0$ \cite{BalMur,Sic},\cite[Exercise 18.2.8]{MilTak}. One uses the explicit formula, rather than mollification, and the proportion $\frac{1}{2}$ arises from the choice of a test function with certain positivity properties.

It seems plausible that one may obtain a larger proportion of non-vanishing by also averaging over moduli $q$. Indeed, Iwaniec and Sarnak \cite{IwaSar} already claimed that by averaging over moduli one can prove at least half of the central values are nonzero. This is striking, in that it is as strong, on average, as the proportion obtained via GRH.

A natural question is whether, by averaging over moduli, one can breach the 50\% barrier, thereby going beyond the immediate reach of GRH. We answer this question in the affirmative.

Let ${ \sideset{}{^*}{\textstyle\sum}_{\chi (q)} }$ denote a sum over the primitive characters modulo $q$, and define $\varphi^*(q)$ to be the number of primitive characters modulo $q$.

\begin{theorem}\label{thm: main theorem}
Let $\Psi$ be a fixed, nonnegative smooth function, compactly supported in $[\frac{1}{2},2]$, which satisfies
\begin{align*}
\int_\mathbb{R} \Psi(x) dx > 0.
\end{align*}
Then for $Q$ sufficiently large we have
\begin{align*}
\sum_{q} \Psi \left( \frac{q}{Q}\right) \frac{q}{\varphi(q)} \ \sideset{}{^*}\sum_{\substack{\chi (q) \\ L \left( \frac{1}{2},\chi\right) \neq 0}} 1 \geq 0.50073 \sum_{q} \Psi \left( \frac{q}{Q}\right) \frac{q}{\varphi(q)} \varphi^*(q).
\end{align*}
\end{theorem}
Thus, roughly speaking, a randomly chosen central value $L \left( \frac{1}{2},\chi\right)$ is more likely nonzero than zero. We remark also that the appearance of the arithmetic weight $\frac{q}{\varphi(q)}$ is technically convenient, but not essential.

\section{Mollification, and a sketch for Theorem \ref{thm: main theorem}}\label{sec: mollify}

The proof of Theorem \ref{thm: main theorem} relies on the powerful technique of mollification. For each character $\chi$ we associate a function $\psi(\chi)$, called a mollifier, that serves to dampen the large values of $L(\frac{1}{2},\chi)$. By the Cauchy-Schwarz inequality we have
\begin{align}\label{eq: outline cauchy schwarz}
\frac{\left|\sum_{q \asymp Q}\sideset{}{^*}\sum_{\chi (q)} L \left( \frac{1}{2},\chi \right) \psi(\chi) \right|^2}{\sum_{q \asymp Q}\sideset{}{^*}\sum_{\chi (q)} \left|L \left( \frac{1}{2},\chi \right) \psi(\chi)\right|^2} \leq \sum_{q \asymp Q} \sideset{}{^*}\sum_{\substack{\chi (q) \\ L \left( \frac{1}{2},\chi\right) \neq 0}} 1.
\end{align}
The better the mollification by $\psi$, the larger proportion of non-vanishing one can deduce.

It is natural to choose $\psi(\chi)$ such that
\begin{align*}
\psi(\chi) \approx \frac{1}{L \left( \frac{1}{2},\chi\right)}.
\end{align*}
Since $L \left( \frac{1}{2},\chi\right)$ can be written as a Dirichlet series
\begin{align}\label{eq: long central value approx}
L \left( \frac{1}{2},\chi\right) = \sum_{n=1}^\infty \frac{\chi(n)}{n^{\frac{1}{2}}},
\end{align}
this suggests the choice
\begin{align}\label{eq: rough IS mollifier}
\psi(\chi) \approx \sum_{\ell \leq y} \frac{\mu(\ell)\chi(\ell)}{\ell^{\frac{1}{2}}}.
\end{align}
We have introduced a truncation $y$ in anticipation of the need to control various error terms that will arise. We write $y = Q^{\theta}$, where $\theta > 0$ is a real number. At least heuristically, larger values of $\theta$ yield better mollification by \eqref{eq: rough IS mollifier}. Iwaniec and Sarnak \cite{IwaSar} made this choice \eqref{eq: rough IS mollifier} (up to some smoothing), and found that the proportion of non-vanishing attained was
\begin{align}\label{eq: IS prop of nonvanishing}
\frac{\theta}{1+\theta}.
\end{align}
When $\theta =1$ we see \eqref{eq: IS prop of nonvanishing} is exactly $\frac{1}{2}$, so we need $\theta > 1$ in order to conclude Theorem \ref{thm: main theorem}. This seems beyond the range of present technology. Without averaging over moduli we may take $\theta = \frac{1}{2} - \varepsilon$, and the asymptotic large sieve of Conrey, Iwaniec, and Soundararajan \cite{CIS} allows one to take $\theta = 1 - \varepsilon$ if one averages over moduli. This just falls short of our goal.

Thus, a better mollifier than \eqref{eq: rough IS mollifier} is required. Part of the problem is that \eqref{eq: long central value approx} is an inefficient representation of $L \left( \frac{1}{2},\chi\right)$. A better representation of $L\left( \frac{1}{2},\chi\right)$ may be obtained through the approximate functional equation, which states
\begin{align}\label{eq: rough app func equation}
L \left( \frac{1}{2},\chi\right) \approx \sum_{n \leq q^{1/2}} \frac{\chi(n)}{n^{\frac{1}{2}}} + \epsilon(\chi) \sum_{n \leq q^{1/2}} \frac{\overline{\chi}(n)}{n^{\frac{1}{2}}}.
\end{align}
Here $\epsilon(\chi)$ is the root number, which is a complex number of modulus 1 defined by
\begin{align}\label{eq: defn of epsilon root number}
\epsilon(\chi) = \frac{1}{q^{\frac{1}{2}}} \sum_{h (\text{mod }q)} \chi(h) e\left( \frac{h}{q}\right).
\end{align}
Inspired by \eqref{eq: rough app func equation}, Michel and VanderKam \cite{MicVan} chose a mollifier
\begin{align}\label{eq: rough MV mollifier}
\psi(\chi) \approx \sum_{\ell \leq y} \frac{\mu(\ell)\chi(\ell)}{\ell^{\frac{1}{2}}} + \overline{\epsilon}(\chi) \sum_{\ell \leq y} \frac{\mu(\ell)\overline{\chi}(\ell)}{\ell^{\frac{1}{2}}}.
\end{align}
We note that Soundararajan \cite{Sou1} earlier used a mollifier of this shape in the context of the Riemann zeta function.

For $y = Q^\theta$, Michel and VanderKam found that \eqref{eq: rough MV mollifier} gives a non-vanishing proportion of
\begin{align}\label{eq: MV prop of nonvanishing}
\frac{2\theta}{1+2\theta}.
\end{align}
Thus, we need $\theta = \frac{1}{2} + \varepsilon$ in order for \eqref{eq: MV prop of nonvanishing} to imply a proportion of non-vanishing greater than $\frac{1}{2}$. However, the more complicated nature of the mollifier \eqref{eq: rough MV mollifier} means that, without averaging over moduli, only the choice $\theta = \frac{3}{10} - \varepsilon$ is acceptable \cite{KhaNgo}.

As we allow ourselves to average over moduli, however, one might hope to obtain \eqref{eq: MV prop of nonvanishing} for $\theta = \frac{1}{2} + \varepsilon$. Again we fall just short of our goal. Using a powerful result of Deshouillers and Iwaniec on cancellation in sums of Kloosterman sums (see Lemma \ref{lem: sum of kloos sum} below) we shall show that $\theta = \frac{1}{2} - \varepsilon$ is acceptable, but increasing $\theta$ any further seems very difficult. It follows that we need any extra amount of mollification in order to obtain a proportion of non-vanishing strictly greater than $\frac{1}{2}$.

The solution is to attach yet another piece to the mollifier $\psi(\chi)$, but here we wish for the mollifier to have a very different shape from \eqref{eq: rough MV mollifier}. Such a mollifier was utilized by Bui \cite{Bui}, who showed that
\begin{align}\label{eq: rough B mollifier}
\psi_{\text{B}}(\chi) \approx \frac{1}{\log q} \mathop{\sum \sum}_{bc \leq y} \frac{\Lambda(b)\mu(c) \overline{\chi}(b) \chi(c)}{(bc)^{\frac{1}{2}}}
\end{align}
is a mollifier for $L \left( \frac{1}{2},\chi\right)$. It turns out that adding \eqref{eq: rough B mollifier} to \eqref{eq: rough MV mollifier} gives a sufficient mollifier to conclude Theorem \ref{thm: main theorem}.

One may roughly motivate a mollifier of the shape \eqref{eq: rough B mollifier} as follows. Working formally,
\begin{align*}
\frac{1}{L(\frac{1}{2},\chi)} &= \frac{L (\frac{1}{2},\overline{\chi})}{L(\frac{1}{2},\chi)L (\frac{1}{2},\overline{\chi})} = \mathop{\sum \sum\sum}_{r,s,v} \frac{\overline{\chi}(r)\mu(s) \overline{\chi}(s) \mu(v) \chi(v)}{(rsv)^{\frac{1}{2}}} \\
&\approx \mathop{\sum \sum\sum}_{r,s,v} \frac{\log r}{\log q} \frac{\overline{\chi}(r)\mu(s) \overline{\chi}(s) \mu(v) \chi(v)}{(rsv)^{\frac{1}{2}}} \\
&= \frac{1}{\log q} \mathop{\sum\sum}_{u,v} \frac{(\mu \star \log)(u) \overline{\chi}(u) \mu(v) \chi(v)}{(uv)^{\frac{1}{2}}}.
\end{align*}

One might wonder what percentage of non-vanishing one can obtain using only a mollifier of the shape \eqref{eq: rough B mollifier}. The analysis for Bui's mollifier is more complicated, and it does not seem possible to write down simple expressions like \eqref{eq: IS prop of nonvanishing} or \eqref{eq: MV prop of nonvanishing} that give a percentage of non-vanishing for \eqref{eq: rough B mollifier} in terms of $\theta$. If one assumes, perhaps optimistically, that averaging over moduli allows one to take any $\theta < 1$ in \eqref{eq: rough B mollifier}, then some numerical computation indicates that the non-vanishing percentage does not exceed 27\%, say.

We remark that, in the course of the proof, the main terms are easily extracted and we have no need here for the averaging over moduli. We require the averaging over moduli in order to estimate some of the error terms.

The structure of the remainder of the paper is as follows. In section \ref{sec: thm first steps} we reduce the proof of Theorem \ref{thm: main theorem} to two technical results, Lemma \ref{lem: main term IS MV} and Lemma \ref{lem: main term B MV}, which give asymptotic evaluations of certain mollified sums. In section \ref{sec: IS MV main term} we extract the main term of Lemma \ref{lem: main term IS MV}, and in section \ref{sec: IS MV error} we use estimates on sums of Kloosterman sums to complete the proof of this lemma. Section \ref{sec: MV B main term} similarly proves the main term of Lemma \ref{lem: main term B MV}, but this derivation is longer than that given in section \ref{sec: IS MV main term} because the main terms are more complicated. In the final section, section \ref{sec: B MV error}, we bound the error term in Lemma \ref{lem: main term B MV}, again using results on sums of Kloosterman sums.

\section{Proof of Theorem \ref{thm: main theorem}: first steps}\label{sec: thm first steps}

Let us fix some notation and conventions that shall hold for the remainder of the paper. 

The notation $a \equiv b (q)$ means $a \equiv b \pmod{q}$, and when $a(q)$ occurs beneath a sum it indicates a summation over residue classes modulo $q$.

We denote by $\epsilon$ an arbitrarily small positive quantity that may vary from one line to the next, or even within the same line. Thus, we may write $X^{2\epsilon} \leq X^\epsilon$ with no reservations.

We need to treat separately the even primitive characters and odd primitive characters. We focus exclusively on the even primitive characters, since the case of odd characters is nearly identical. We write $\sideset{}{^+}{\textstyle\sum}_{\chi (q)}$ for a sum over even primitive characters modulo $q$, and we write $\varphi^+(q)$ for the number of such characters. Observe that $\varphi^+(q) = \frac{1}{2}\varphi^*(q) + O(1)$.

We shall encounter the Ramanujan sum $c_q(n)$ (see the proof of Proposition \ref{prop: cancellation in n}), defined by
\begin{align*}
c_q(n) &= \sum_{\substack{a(q) \\ (a,q)=1}} e \left(\frac{an}{q} \right).
\end{align*}
We shall only need to know that $c_q(1) = \mu(q)$ and $|c_q(n)| \leq (q,n)$, where $(q,n)$ is the greatest common divisor of $q$ and $n$.

We now fix a smooth function $\Psi$ as in the statement of Theorem \ref{thm: main theorem}, and allow all implied constants to depend on $\Psi$. We let $Q$ be a large real number, and set $y_i = Q^{\theta_i}$ for $i \in \{1,2,3\}$, where $0 < \theta_i < \frac{1}{2}$ are fixed real numbers. We further define $L = \log Q$. The notation $o(1)$ denotes a quantity that goes to zero as $Q$ goes to infinity.

Let us now begin the proof of Theorem \ref{thm: main theorem} in earnest.  As discussed in section \ref{sec: mollify}, we choose our mollifier $\psi(\chi)$ to have the form 
\begin{align}\label{eq: three piece mollifier}
\psi(\chi) = \psi_{\text{IS}}(\chi)  + \psi_{\text{B}}(\chi)+ \psi_{\text{MV}}(\chi),
\end{align}
where
\begin{align}\label{eq: defn of mollifiers}
\psi_{\text{IS}}(\chi) &= \sum_{\ell \leq y_1} \frac{\mu(\ell)}{\ell^{\frac{1}{2}}} P_1 \left(\frac{\log(y_1/\ell)}{\log y_1} \right), \nonumber \\
\psi_{\text{B}}(\chi) &= \frac{1}{L}\mathop{\sum \sum}_{bc \leq y_2} \frac{\Lambda(b)\mu(c) \overline{\chi}(b) \chi(c)}{(bc)^{\frac{1}{2}}} P_2 \left(\frac{\log(y_2/bc)}{\log y_2} \right), \\
\psi_{\text{MV}}(\chi) &= \epsilon(\overline{\chi}) \sum_{\ell \leq y_3} \frac{\mu(\ell)\overline{\chi}(\ell)}{\ell^{\frac{1}{2}}}P_3 \left(\frac{\log(y_3/\ell)}{\log y_3} \right). \nonumber
\end{align}
The smoothing polynomials $P_i$ are real and satisfy $P_i(0)=0$. For notational convenience we write
\begin{align*}
P_i \left(\frac{\log(y_i/x)}{\log y_i} \right) = P_i[x].
\end{align*}
There is some ambiguity in this notation because of the $y_i$-dependence in the polynomials, and this needs to be remembered in calculations.

Now define sums $S_1$ and $S_2$ by
\begin{align}\label{eq: defn of S1 and S2}
S_1 &= \sum_{q} \Psi \left( \frac{q}{Q}\right) \frac{q}{\varphi(q)} \ \sideset{}{^+}\sum_{\substack{\chi (q)}} L \left( \frac{1}{2},\chi\right) \psi(\chi), \\
S_2 &= \sum_{q} \Psi \left( \frac{q}{Q}\right) \frac{q}{\varphi(q)} \ \sideset{}{^+}\sum_{\substack{\chi (q)}} \left|L \left( \frac{1}{2},\chi\right) \psi(\chi)\right|^2. \nonumber
\end{align}
We apply Cauchy-Schwarz as in \eqref{eq: outline cauchy schwarz} and get
\begin{align}\label{eq: cauchy for nonvanishing}
\sum_{q} \Psi \left( \frac{q}{Q}\right) \frac{q}{\varphi(q)} \ \sideset{}{^+}\sum_{\substack{\chi (q) \\ L \left( \frac{1}{2},\chi\right) \neq 0}} 1 &\geq \frac{S_1^2}{S_2}.
\end{align}
The proof of Theorem \ref{thm: main theorem} therefore reduces to estimating $S_1$ and $S_2$. We obtain asymptotic formulas for these two sums.

\begin{lemma}\label{lem: eval of S1}
Suppose $0 < \theta_1, \theta_2 < 1$ and $0 < \theta_3 < \frac{1}{2}$. Then
\begin{align*}
S_1 = \left(P_1(1) + P_3(1) +\frac{\theta_2}{2}\widetilde{P_2}(1) + o(1) \right) \sum_{q} \Psi \left( \frac{q}{Q}\right) \frac{q}{\varphi(q)} \varphi^+(q),
\end{align*}
where
\begin{align*}
\widetilde{P_2}(x) = \int_0^x P_2(u) du.
\end{align*}
\end{lemma}

\begin{lemma}\label{lem: eval of S2}
Let $0 <\theta_1,\theta_2,\theta_3 < \frac{1}{2}$ with $\theta_2 < \theta_1,\theta_3$. Then
\begin{align*}
S_2 &= \left( 2P_1(1)P_3(1) + P_3(1)^2 + \frac{1}{\theta_3}\int_0^1 P_3'(x)^2 dx + \kappa + \lambda + o(1) \right)\sum_{q} \Psi \left( \frac{q}{Q}\right) \frac{q}{\varphi(q)} \varphi^+(q),
\end{align*}
where
\begin{align*}
\kappa = 3\theta_2 P_3(1)\widetilde{P_2}(1) - 2\theta_2 \int_0^1 P_2(x)P_3(x)dx
\end{align*}
and
\begin{align*}
\lambda &= P_1(1)^2 + \frac{1}{\theta_1}\int_0^1 P_1'(x)^2 dx - \theta_2 P_1(1) \widetilde{P_2}(1) + 2\theta_2\int_0^1 P_1\left(1 - \frac{\theta_2(1-x)}{\theta_1}\right) P_2(x) dx \\
&+ \frac{\theta_2}{\theta_1} \int_0^1 P_1'\left(1 - \frac{\theta_2(1-x)}{\theta_1}\right) P_2(x) dx + \theta_2^2\int_0^1 (1-x)P_2(x)^2 dx \\ 
&+ \frac{\theta_2}{2}\int_0^1 (1-x)^2 P_2'(x)^2 dx - \frac{\theta_2^2}{4} \widetilde{P_2}(1)^2 + \frac{\theta_2}{4}\int_0^1 P_2(x)^2 dx.
\end{align*}
\end{lemma}

\begin{proof}[Proof of Theorem \ref{thm: main theorem}]
Lemmas \ref{lem: eval of S1} and \ref{lem: eval of S2} give the evaluations of $S_1$ and $S_2$ for even characters. The identical formulas hold for odd characters. Theorem \ref{thm: main theorem} then follows from \eqref{eq: cauchy for nonvanishing} upon choosing $\theta_1 =\theta_3= \frac{1}{2}$, $\theta_2 = 0.163$, and
\begin{align*}
P_1(x) &= 4.86x + 0.29x^2 - 0.96x^3 + 0.974x^4 - 0.17x^5, \\
P_2(x) &= -3.11x - 0.3x^2 + 0.87x^3 - 0.18x^4 - 0.53x^5, \\
P_3(x) &= 4.86x + 0.06x^2.
\end{align*}
These choices actually yield a proportion\footnote{\label{foot: mathematica notebook}A Mathematica\textregistered\ file with this computation is included with this paper on \text{arxiv.org}.}
\begin{align*}
\geq 0.50073004\ldots,
\end{align*}
which allows us to state Theorem \ref{thm: main theorem} with a clean inequality.
\end{proof}

We note without further comment the curiosity in the proof of Theorem \ref{thm: main theorem} that the largest permissible value of $\theta_2$ is not optimal.

We can dispense with $S_1$ quickly.
\begin{proof}[Proof of Lemma \ref{lem: eval of S1}]
Apply \cite[Theorem 2.1]{Bui} and the argument of \cite[section 3]{MicVan}, using the facts $L = \log q + O(1)$ and $y_i = q^{\theta_i + o(1)}$.
\end{proof}

The analysis of $S_2$ is much more involved, and we devote the remainder of the paper to this task. We first observe that \eqref{eq: three piece mollifier} yields
\begin{align*}
|\psi(\chi)|^2 &= |\psi_{\text{IS}}(\chi) + \psi_{\text{B}}(\chi)|^2 + 2\text{Re}\left\{\psi_{\text{IS}}(\chi) \psi_{\text{MV}}(\overline{\chi}) + \psi_{\text{B}}(\chi) \psi_{\text{MV}}(\overline{\chi})\right\} + |\psi_{\text{MV}}(\chi)|^2.
\end{align*}
By \cite[Theorem 2.2]{Bui} we have
\begin{align*}
\sideset{}{^+}\sum_{\chi(q)} \left| L \left( \frac{1}{2},\chi \right)\right|^2|\psi_{\text{IS}}(\chi) + \psi_{\text{B}}(\chi)|^2 =\lambda \varphi^+(q) + O \left(q L^{-1+\epsilon} \right),
\end{align*}
where $\lambda$ is as in Lemma \ref{lem: eval of S2}. We also have
\begin{align*}
\frac{1}{\varphi^+(q)}\sideset{}{^+}\sum_{\chi(q)} \left| L \left( \frac{1}{2},\chi \right)\right|^2|\psi_{\text{MV}}(\chi)|^2 &= \frac{1}{\varphi^+(q)} \sideset{}{^+}\sum_{\chi(q)} \left| L \left( \frac{1}{2},\chi \right)\right|^2 \left|\sum_{\ell \leq y_3}\frac{\mu(\ell)\chi(\ell) P_3[\ell]}{\ell^{\frac{1}{2}}} \right|^2 \\
&= P_3(1)^2 + \frac{1}{\theta_3} \int_0^1 P_3'(x)^2 dx + O(L^{-1+\epsilon}),
\end{align*}
by the analysis of the Iwaniec-Sarnak mollifier (see \cite[section 2.3]{Bui}). 

Therefore, in order to prove Lemma \ref{lem: eval of S2} it suffices to prove the following two results.

\begin{lemma}\label{lem: main term IS MV}
For $0 < \theta_1,\theta_3 < \frac{1}{2}$ we have
\begin{align*}
\sum_{q} &\Psi \left( \frac{q}{Q}\right) \frac{q}{\varphi(q)}\sideset{}{^+}\sum_{\chi(q)} \left| L \left( \frac{1}{2},\chi \right)\right|^2 \psi_{\text{IS}}(\chi) \psi_{\text{MV}}(\overline{\chi}) \\ 
&= (P_1(1)P_3(1) + o(1))\sum_{q} \Psi \left( \frac{q}{Q}\right) \frac{q}{\varphi(q)} \varphi^+(q). \nonumber
\end{align*}
\end{lemma}

\begin{lemma}\label{lem: main term B MV}
Let $0 < \theta_2 < \theta_3 < \frac{1}{2}$. Then
\begin{align*}
\sum_{q} &\Psi \left( \frac{q}{Q}\right) \frac{q}{\varphi(q)}\sideset{}{^+}\sum_{\chi(q)} \left| L \left( \frac{1}{2},\chi \right)\right|^2 \psi_{\text{B}}(\chi) \psi_{\text{MV}}(\overline{\chi}) \\ 
&= \left(\frac{3\theta_2}{2} P_3(1)\widetilde{P_2}(1) - \theta_2 \int_0^1 P_2(x)P_3(x)dx + o(1)\right)\sum_{q} \Psi \left( \frac{q}{Q}\right) \frac{q}{\varphi(q)} \varphi^+(q).
\end{align*}
\end{lemma}

\section{Lemma \ref{lem: main term IS MV} main term}\label{sec: IS MV main term}

The goal of this section is to extract the main term in Lemma \ref{lem: main term IS MV}. The main term analysis is given in \cite[section 6]{MicVan}, but as the ideas also appear in the proof of Lemma \ref{lem: main term B MV} we give details here.

We begin with two lemmas.

\begin{lemma}\label{lem: approx func eq}
Let $\chi$ be a primitive even character modulo $q$. Let $G(s)$ be an even polynomial satisfying $G(0) = 1$, and which vanishes to second order at $\frac{1}{2}$. Then we have
\begin{align*}
\left| L \left( \frac{1}{2},\chi \right) \right|^2 &= 2 \mathop{\sum \sum}_{m,n} \frac{\chi(m) \overline{\chi}(n)}{(mn)^{\frac{1}{2}}} V \left(\frac{mn}{q} \right),
\end{align*}
where
\begin{align}\label{eq: defn of V}
V(x) &= \frac{1}{2\pi i}\int_{(1)} \frac{\Gamma^2 \left( \frac{s}{2} + \frac{1}{4}\right)}{\Gamma^2 \left( \frac{1}{4}\right)} \frac{G(s)}{s} \pi^{-s}x^{-s}ds.
\end{align}
\end{lemma}
\begin{proof}
See \cite[(2.5)]{IwaSar}. The result follows along the lines of \cite[Theorem 5.3]{IwaKow}.
\end{proof}
We remark that $V$ satisfies $V(x) \ll_A (1+x)^{-A}$, as can be seen by moving the contour of integration to the right. We also note that the choice of $G(s)$ in Lemma \ref{lem: approx func eq} is almost completely free. In particular, we may choose $G$ to vanish at whichever finite set of points is convenient for us (see \eqref{eq: F has rapid decay} below for an application).

\begin{lemma}\label{lem: character orthogonality}
Let $(mn,q)=1$. Then
\begin{align*}
\sideset{}{^+}\sum_{\chi (q)} \chi(m) \overline{\chi}(n) = \frac{1}{2}\mathop{\sum \sum}_{\substack{vw = q \\ w \mid m \pm n}} \mu(v) \varphi(w).
\end{align*}
\end{lemma}
\begin{proof}
See \cite[Lemma 4.1]{BM}, for instance.
\end{proof}

We do not need the averaging over $q$ in order to extract the main term of Lemma \ref{lem: main term IS MV}. We insert the definitions of the mollifiers $\psi_{\text{IS}}(\chi)$ and $\psi_{\text{MV}}(\overline{\chi})$, then apply Lemma \ref{lem: approx func eq} and interchange orders of summation. We obtain
\begin{align}\label{eq: ISMV all terms}
&\sideset{}{^+}\sum_{\chi(q)} \left| L \left( \frac{1}{2},\chi \right)\right|^2 \psi_{\text{IS}}(\chi) \psi_{\text{MV}}(\overline{\chi}) \\ 
&= 2\mathop{\sum \sum}_{\substack{\ell_1 \leq y_1 \\ \ell_3 \leq y_3 \\ (\ell_1\ell_3,q)=1}} \frac{\mu(\ell_1) \mu(\ell_3)P_1[\ell_1]P_3[\ell_3]}{(\ell_1\ell_3)^{\frac{1}{2}}} \mathop{\sum \sum}_{(mn,q)=1} \frac{1}{(mn)^{\frac{1}{2}}} V \left( \frac{mn}{q}\right) \sideset{}{^+}\sum_{\chi(q)}\epsilon(\chi) \chi(m\ell_1 \ell_3) \overline{\chi}(n). \nonumber
\end{align}
Opening $\epsilon(\chi)$ using \eqref{eq: defn of epsilon root number} and applying Lemma \ref{lem: character orthogonality}, we find after some work (see \cite[(3.4) and (3.7)]{IwaSar}) that
\begin{align}\label{eq: eps char orth}
\sideset{}{^+}\sum_{\chi(q)}\epsilon(\chi) \chi(m\ell_1 \ell_3) \overline{\chi}(n) &= \frac{1}{q^{1/2}}\mathop{\sum \sum}_{\substack{vw = q \\ (v,w)=1}} \mu^2(v)\varphi(w) \cos \left(\frac{2\pi n \overline{m \ell_1 \ell_3 v}}{w} \right).
\end{align}
The main term comes from $m\ell_1\ell_3 = 1$. With this constraint in place we apply character orthogonality in reverse, obtaining that the main term $M_{\text{IS,MV}}$ of Lemma  \ref{lem: main term IS MV} is
\begin{align*}
M_{\text{IS,MV}} &= 2P_1(1) P_3(1) \sideset{}{^+}\sum_{\chi(q)} \epsilon(\chi) \sum_n \frac{\overline{\chi}(n)}{n^{1/2}} V \left( \frac{n}{q}\right).
\end{align*}

We have the following proposition.
\begin{proposition}\label{prop: change V for F}
Let $\chi$ be a primitive even character modulo $q$, and let $T > 0$ be a real number. Let $V$ be defined as in \eqref{eq: defn of V}. Then
\begin{align*}
\sum_n \frac{\overline{\chi}(n)}{n^{\frac{1}{2}}} V \left( \frac{Tn}{q}\right) &= L \left( \frac{1}{2},\overline{\chi}\right) - \epsilon(\overline{\chi}) \sum_n \frac{\chi(n)}{n^{\frac{1}{2}}} F \left( \frac{n}{T} \right),
\end{align*}
where
\begin{align}\label{eq: defn of func F}
F(x) &= \frac{1}{2\pi i} \int_{(1)} \frac{\Gamma\left( \frac{s}{2} + \frac{1}{4}\right)\Gamma\left(- \frac{s}{2} + \frac{1}{4}\right)}{\Gamma^2\left(\frac{1}{4}\right)} \frac{G(s)}{s} x^{-s}ds.
\end{align}
\end{proposition}

Before proving Proposition \ref{prop: change V for F}, let us see how to use it to finish the evaluation of $M_{\text{IS,MV}}$. Proposition \ref{prop: change V for F} gives
\begin{align*}
M_{\text{IS,MV}} &= 2P_1(1)P_3(1)\sideset{}{^+}\sum_{\chi(q)} \epsilon(\chi) L \left( \frac{1}{2},\overline{\chi}\right) - 2P_1(1)P_3(1) \sideset{}{^+}\sum_{\chi(q)} \sum_n \frac{\chi(n)}{n^{\frac{1}{2}}} F(n),
\end{align*}
and by the first moment analysis (see \cite[section 3]{MicVan}, also section \ref{sec: MV B main term} below) we have
\begin{align}\label{eq: IS MV pos mt}
2P_1(1) P_3(1) \sideset{}{^+}\sum_{\chi(q)} \epsilon(\chi)L \left( \frac{1}{2},\overline{\chi}\right) = (1+o(1))2 P_1(1) P_3(1)\varphi^+(q).
\end{align}
For the other piece, we apply Lemma \ref{lem: character orthogonality} to obtain
\begin{align*}
- 2P_1(1) P_3(1) \sideset{}{^+}\sum_{\chi(q)} \sum_n \frac{\chi(n)}{n^{\frac{1}{2}}} F(n) = -P_1(1) P_3(1)\sum_{w \mid q} \varphi(w) \mu(q/w)\sum_{\substack{n \equiv \pm 1 (w) \\ (n,q)=1}} \frac{1}{n^{\frac{1}{2}}} F(n).
\end{align*}
We choose $G$ to vanish at all the poles of
\begin{align*}
\Gamma\left( \frac{s}{2} + \frac{1}{4}\right)\Gamma\left(- \frac{s}{2} + \frac{1}{4}\right)
\end{align*}
in the disc $|s| \leq A$, where $A>0$ is large but fixed. By moving the contour of integration to the right we see
\begin{align}\label{eq: F has rapid decay}
F(x) \ll \frac{1}{(1+x)^{100}},
\end{align}
say, and therefore the contribution from $n > q^{\frac{1}{10}}$ is negligible. By trivial estimation the contribution from $w \leq q^{\frac{1}{4}}$ is also negligible. For $w > q^{\frac{1}{4}}$ and $n \leq q^{\frac{1}{10}}$, we can only have $n \equiv \pm 1 \pmod{w}$ if $n = 1$. Adding back in the terms with $n \leq q^{\frac{1}{4}}$, the contribution from these terms is therefore
\begin{align}\label{eq: IS MV negative mt}
-(1+o(1))2P_1(1) P_3(1)F(1) \varphi^+(q).
\end{align}
Since the integrand in $F(1)$ is odd, we may evaluate $F(1)$ through a residue at $s = 0$. We shift the line of integration in \eqref{eq: defn of func F} to $\text{Re}(s) = -1$, picking up a contribution from the simple pole at $s = 0$. In the integral on the line $\text{Re}(s) = -1$ we change variables $s \rightarrow -s$. This yields the relation $F(1) = 1-F(1)$, whence $F(1) = \frac{1}{2}$. Combining \eqref{eq: IS MV pos mt} and \eqref{eq: IS MV negative mt}, we obtain
\begin{align*}
M_{\text{IS,MV}} &= (1+o(1)) P_1(1) P_3(1) \varphi^+(q),
\end{align*}
as desired. This yields the main term of Lemma \ref{lem: main term IS MV}.

\begin{proof}[Proof of Proposition \ref{prop: change V for F}]
We write $V$ using its definition and interchange orders of summation and integration to get
\begin{align*}
\sum_n \frac{\overline{\chi}(n)}{n^{1/2}} V \left( \frac{Tn}{q}\right) &= \frac{1}{2\pi i}\int_{(1)} \frac{\Gamma^2 \left( \frac{s}{2} + \frac{1}{4}\right)}{\Gamma^2 \left( \frac{1}{4}\right)} \frac{G(s)}{s} \left( \frac{q}{\pi}\right)^s T^{-s} L \left( \frac{1}{2} + s, \overline{\chi}\right) ds.
\end{align*}
We move the line of integration to $\text{Re}(s) = -1$, picking up a contribution of $L \left( \frac{1}{2},\overline{\chi}\right)$ from the pole at $s = 0$. Observe that we do not get any contribution from the double pole of $\Gamma^2 \left( \frac{s}{2} + \frac{1}{4}\right)$ at $s = -\frac{1}{2}$ because of our assumption that $G$ vanishes at $s = \pm\frac{1}{2}$ to second order.

Now, for the integral on the line $\text{Re}(s) = -1$, we apply the functional equation for $L \left( \frac{1}{2} + s, \overline{\chi}\right)$ and then change variables $s \rightarrow -s$ to obtain
\begin{align*}
-\epsilon(\overline{\chi})\frac{1}{2\pi i}\int_{(1)} \frac{\Gamma\left( \frac{s}{2} + \frac{1}{4}\right)\Gamma\left(- \frac{s}{2} + \frac{1}{4}\right)}{\Gamma^2\left(\frac{1}{4}\right)} \frac{G(s)}{s} T^s L \left( \frac{1}{2}+s,\chi\right) ds.
\end{align*}
The desired result follows by expanding $L\left( \frac{1}{2} + s,\chi\right)$ in its Dirichlet series and interchanging the order of summation and integration.
\end{proof}

\section{Lemma \ref{lem: main term IS MV}: error term}\label{sec: IS MV error}

Here we show that the remainder of the terms in \eqref{eq: ISMV all terms} (those with $m\ell_1\ell_3 \neq 1$) contribute only to the error term of Lemma \ref{lem: main term IS MV}. Here we must avail ourselves of the averaging over $q$.

Inserting \eqref{eq: eps char orth} into \eqref{eq: ISMV all terms} and averaging over moduli, we wish to show that
\begin{align}\label{eq: desired E1 bound}
\mathcal{E}_{1} &= \mathop{\sum \sum}_{(v,w)=1} \mu^2(v) \frac{v}{\varphi(v)} \frac{w^{\frac{1}{2}}}{v^{\frac{1}{2}}} \Psi\left(\frac{vw}{Q} \right)\mathop{\sum \sum}_{\substack{\ell_1 \leq y_1 \\ \ell_3 \leq y_3 \\ (\ell_1\ell_3,vw)=1}} \frac{\mu(\ell_1) \mu(\ell_3)P_1[\ell_1]P_3[\ell_3]}{(\ell_1\ell_3)^{\frac{1}{2}}} \\ 
&\times\mathop{\sum \sum}_{(mn,vw)=1} \frac{1}{(mn)^{\frac{1}{2}}} Z \left( \frac{mn}{vw}\right) \cos \left( \frac{2\pi n \overline{m\ell_1\ell_3 v}}{w} \right) \ll Q^{2-\epsilon + o(1)}, \nonumber
\end{align}
where $m \ell_1 \ell_3 \neq 1$, but we do not indicate this in the notation. The function $Z$ is actually just $V$ in \eqref{eq: defn of V}, but we do not wish to confuse the function $V$ with the scale $V$ that shall appear shortly.

Observe that the arithmetic weight $\frac{q}{\varphi(q)}$ has become $\frac{v}{\varphi(v)}\frac{w}{\varphi(w)}$ by multiplicativity, and that this factor of $\varphi(w)$ has canceled with $\varphi(w)$ in \eqref{eq: eps char orth}, making the sum on $w$ smooth.

The main tool we use to bound $\mathcal{E}_1'$ is the following result, due to Deshouillers and Iwaniec, on cancellation in sums of Kloosterman sums.
\begin{lemma}\label{lem: sum of kloos sum}
Let $C,D,N,R,S$ be positive numbers, and let $b_{n,r,s}$ be a complex sequence supported in $(0,N] \times (R,2R] \times (S,2S] \cap \mathbb{N}^3$. Let $g_0(\xi,\eta)$ be a smooth function having compact support in $\mathbb{R}^+ \times \mathbb{R}^+$, and let $g(c,d) = g_0(c/C,d/D)$. Then
\begin{align*}
\mathop{\sum_c \sum_d \sum_n \sum_r \sum_s}_{(rd,sc)=1} &b_{n,r,s} g(c,d) e \left(n \frac{\overline{rd}}{sc}\right) \\ 
&\ll_{\epsilon,g_0} (CDNRS)^{\epsilon} K(C,D,N,R,S)\| b_{N,R,S}\|_2,
\end{align*}
where
\begin{align*}
\| b_{N,R,S}\|_2 = \left(\sum_n \sum_r\sum_s |b_{n,r,s}|^2 \right)^{\frac{1}{2}}
\end{align*}
and
\begin{align*}
K (C,D,N,R,S)^2 = CS(RS+N)(C+RD) + C^2DS \sqrt{(RS+N)R} + D^2 N R S^{-1}.
\end{align*}
\end{lemma}
\begin{proof}
This is essentially \cite[Lemma 1]{BFI}, which is an easy consequence of \cite[Theorem 12]{DesIwa}.
\end{proof}

We need to massage \eqref{eq: desired E1 bound} before it is in a form where an application of Lemma \ref{lem: sum of kloos sum} is appropriate. Let us briefly describe our plan of attack. We apply partitions of unity to localize the variables and then separate variables with integral transforms. By using the orthogonality of multiplicative characters we will be able to assume that $v$ is quite small, which is advantageous when it comes time to remove coprimality conditions involving $v$. We next reduce to the case in which $n$ is somewhat small. This is due to the fact that the sum on $n$ is essentially a Ramanujan sum, and Ramanujan sums experience better than squareroot cancellation on average. We next use M\"obius inversion to remove the coprimality condition between $n$ and $w$. This application of M\"obius inversion introduces a new variable, call if $f$, and another application of character orthogonality allows us to assume $f$ is small. We then remove the coprimality conditions on $m$. We finally apply Lemma \ref{lem: sum of kloos sum} to get the desired cancellation, and it is crucial here that $f$ and $v$ are no larger than $Q^\epsilon$.

Let us turn to the details in earnest. We apply smooth partitions of unity (see \cite[Lemma 1.6]{BFKMM}, for instance) in all variables, so that $\mathcal{E}_1$ can be written
\begin{align}\label{eq: E1 after parts of unity}
\mathop{\sum \cdots \sum}_{M,N,L_1,L_3,V,W} \mathcal{E}_1 (M,N,L_1,L_3,V,W),
\end{align}
where
\begin{align*}
\mathcal{E}_1 (M,N,L_1,L_3,V,W) &= \mathop{\sum \sum}_{(v,w)=1} \mu^2(v) \frac{v}{\varphi(v)} \frac{w^{\frac{1}{2}}}{v^{\frac{1}{2}}}\Psi\left(\frac{vw}{Q} \right)G \left( \frac{v}{V}\right) G \left( \frac{w}{W}\right) \\
&\times\mathop{\sum \sum}_{\substack{\ell_1\leq y_1 \\ \ell_3 \leq y_3 \\ (\ell_1\ell_3,vw)=1}} \frac{\mu(\ell_1) \mu(\ell_3)P_1[\ell_1]P_3[\ell_3]}{(\ell_1\ell_3)^{\frac{1}{2}}}G \left( \frac{\ell_1}{L_1}\right) G \left( \frac{\ell_3}{L_3}\right) \\ 
&\times\mathop{\sum \sum}_{(mn,vw)=1} \frac{1}{(mn)^{\frac{1}{2}}} Z \left( \frac{mn}{vw}\right) G \left( \frac{m}{M}\right) G \left( \frac{n}{N}\right) \cos \left( \frac{2\pi n \overline{m\ell_1\ell_3 v}}{w} \right).
\end{align*}
Here $G$ is a smooth, nonnegative function supported in $[\frac{1}{2},2]$, and the numbers $M,N,L_i,V,W$ in \eqref{eq: E1 after parts of unity} range over powers of two. We may assume 
\begin{align*}
M,N,L_1,L_3,V,W \gg 1, \ \ \ \ \ VW \asymp Q, \ \ \ \ \ L_i \ll y.
\end{align*}
Furthermore, by the rapid decay of $Z$ we may assume $MN \leq Q^{1+\epsilon}$. Thus, the number of summands $\mathcal{E}_1(M,\ldots,W)$ in \eqref{eq: E1 after parts of unity} is $\ll Q^{o(1)}$. 

Up to changing the definition of $G$, we may rewrite $\mathcal{E}_1 (M,\cdots,W)$ as
\begin{align*}
\mathcal{E}_1 (M,N,L_1,L_3,V,W) &= \frac{W^{\frac{1}{2}}}{(MNL_1L_3V)^{\frac{1}{2}}} \mathop{\sum \sum}_{(v,w)=1} \alpha(v) G \left( \frac{v}{V}\right) G \left( \frac{w}{W}\right) \Psi \left( \frac{vw}{Q}\right) \\
&\times \mathop{\sum \sum}_{\substack{\ell_i \leq y_i \\ (\ell_i,vw)=1}} \beta(\ell_1)\gamma(\ell_3) G \left( \frac{\ell_1}{L_1}\right) G \left( \frac{\ell_3}{L_3}\right) \\
&\times \mathop{\sum \sum}_{(mn,vw)=1}Z \left( \frac{mn}{vw}\right) G \left( \frac{m}{M}\right) G \left( \frac{n}{N}\right) \cos \left(\frac{2\pi n \overline{m\ell_1\ell_3 v}}{w} \right),
\end{align*}
where $\alpha,\beta,\gamma$ are sequences satisfying $|\alpha(v)|,|\beta(\ell_1)|,|\gamma(\ell_3)| \ll Q^{o(1)}$.

We separate the variables in $Z$ by writing $Z$ using its definition as an integral \eqref{eq: defn of V} and moving the line of integration to $\text{Re}(s) = L^{-1}$. By the rapid decay of the $\Gamma$ function in vertical strips we may restrict to $|\text{Im}(s)| \leq Q^\epsilon$. We similarly separate the variables in $\Psi$ using the inverse Mellin transform. Therefore, up to changing the definition of some of the functions $G$, it suffices to prove that
\begin{align}\label{eq: desired bound for E1'}
\mathcal{E}_{1}' (M,N,L_1,L_3,V,W) &= \frac{W^{\frac{1}{2}}}{(MNL_1L_3V)^{\frac{1}{2}}} \mathop{\sum \sum}_{(v,w)=1} \alpha(v) G \left( \frac{v}{V}\right) G \left( \frac{w}{W}\right) \nonumber \\
&\times \mathop{\sum \sum}_{\substack{\ell_i \leq y_i \\ (\ell_i,vw)=1}} \beta(\ell_1)\gamma(\ell_3) G \left( \frac{\ell_3}{L_3}\right) G \left( \frac{\ell_3}{L_3}\right) \\
&\times \mathop{\sum \sum}_{(mn,vw)=1} G \left( \frac{m}{M}\right) G \left( \frac{n}{N}\right) e \left( \frac{n \overline{m\ell_1\ell_3 v}}{w} \right) \ll  Q^{2-\epsilon + o(1)}.\nonumber
\end{align}
Our smooth functions $G$ all satisfy $G^{(j)}(x) \ll_j Q^{j\epsilon}$ for $j \geq 0$. To save on space we write the left side of \eqref{eq: desired bound for E1'} as simply $\mathcal{E}_1'$.

Observe that the trivial bound for $\mathcal{E}_1'$ is
\begin{align}\label{eq: E1' trivial bound}
\mathcal{E}_1' &\ll V^{\frac{1}{2}}W^{\frac{3}{2}} (MN)^{\frac{1}{2}} (L_1L_3)^{\frac{1}{2}} Q^{o(1)} \ll \frac{Q^{2+\epsilon}(y_1y_3)^{\frac{1}{2}}}{V}.
\end{align}
This bound is worst when $V$ is small. Since $y_i$ will be taken close to $Q^{\frac{1}{2}}$, we therefore need to save $\approx Q^{\frac{1}{2}}$ in order to obtain \eqref{eq: desired E1 bound}. The trivial bound does show, however, that the contribution from $V > Q^{\frac{1}{2}+2\epsilon}$ is acceptably small, and we may therefore assume that $V \leq Q^{\frac{1}{2}+2\epsilon}$. Note this implies $W \gg Q^{\frac{1}{2} - \epsilon}$.

We now reduce to the case $V \ll Q^\epsilon$. We accomplish this by re-introducing multiplicative characters. The orthogonality of multiplicative characters yields
\begin{align}\label{eq: add char to mult char}
e \left( \frac{n \overline{m\ell_1\ell_3 v}}{w} \right) &= \frac{1}{\varphi(w)}\sum_{\chi (w)} \tau(\overline{\chi})\chi(n) \overline{\chi}(m\ell_1\ell_3 v).
\end{align}
Using the Gauss sum bound $|\tau(\overline{\chi})| \ll w^{\frac{1}{2}}$ we then arrange $\mathcal{E}_1'$ as
\begin{align*}
\mathcal{E}_1' &\ll \frac{W}{(MNL_1L_3V)^{\frac{1}{2}}} \sum_{w \asymp W} \frac{1}{\varphi(w)} \sum_{v \asymp V} \left|\mathop{\sum \sum}_{(mn,v)=1} \chi(n) \overline{\chi}(m) \right| \left|\mathop{\sum \sum}_{(\ell_1\ell_3,v)=1} \overline{\chi}(\ell_1\ell_3) \right|,
\end{align*}
where we have suppressed some things in the notation for brevity. By Cauchy-Schwarz and character orthogonality we obtain
\begin{align*}
\sum_{\chi(w)}\left|\mathop{\sum\sum}_{m,n} \right| \left|\mathop{\sum \sum}_{\ell_1,\ell_3} \right| \ll Q^{o(1)} (MNL_1L_3)^{\frac{1}{2}} (MN + W)^{\frac{1}{2}} (L_1L_3 + W)^{\frac{1}{2}},
\end{align*}
which yields a bound of
\begin{align}\label{eq: E1' mult char orthog}
Q^{-o(1)}\mathcal{E}_1' &\ll \frac{Q (MN)^{\frac{1}{2}} (y_1y_3)^{\frac{1}{2}}}{V^{\frac{1}{2}}} + \frac{Q^{\frac{3}{2}} (MN)^{1/2}}{V} + \frac{Q^{\frac{3}{2}} (y_1y_3)^{\frac{1}{2}}}{V} + \frac{Q^2}{V^{\frac{3}{2}}}.
\end{align}
We observe that \eqref{eq: E1' mult char orthog} is acceptable for $V \geq Q^{3\epsilon}$, say. We may therefore assume $V \leq Q^\epsilon$.

We next show that $\mathcal{E}_1'$ is small provided $N$ is somewhat large.
\begin{proposition}\label{prop: cancellation in n}
Assume the hypotheses of Lemma \ref{lem: main term IS MV}. If $N \geq M Q^{-2\epsilon}$ and $m\ell_1\ell_3 \neq 1$, then $\mathcal{E}_1' \ll Q^{2-\epsilon+o(1)}$.
\end{proposition}
\begin{proof}
We make use only of cancellation in the sum on $n$, say
\begin{align*}
\Sigma_N &= \sum_{(n,vw)=1} G\left( \frac{n}{N} \right) e \left( \frac{n \overline{m\ell_1\ell_3 v}}{w} \right).
\end{align*}
We use M\"obius inversion to detect the condition $(n,v)=1$, and then break $n$ into primitive residue classes modulo $w$. Thus
\begin{align*}
\Sigma_N &= \sum_{d \mid v} \mu(d) \sum_{(a,w)=1} e \left( \frac{ad \overline{m\ell_1\ell_3 v}}{w} \right) \sum_{n \equiv a (w)} G \left( \frac{dn}{N}\right).
\end{align*}
We apply Poisson summation to each sum on $n$, and obtain
\begin{align*}
\Sigma_N &= \sum_{d \mid v} \mu(d) \sum_{(a,w)=1} e \left( \frac{ad \overline{m\ell_1\ell_3 v}}{w} \right) \frac{N}{dw}\sum_{|h| \leq W^{1+\epsilon}d/N} e \left(\frac{ah}{w} \right) \widehat{G} \left(\frac{hN}{dw} \right) + O_\epsilon \left(Q^{-100} \right),
\end{align*}
say. The contribution of the error term is, of course, negligible. The contribution of the zero frequency $h = 0$ to $\Sigma_N$ is
\begin{align*}
\widehat{G}(0)\frac{N}{w}\sum_{d \mid v} \frac{\mu(d)}{d}\sum_{(a,w)=1} e \left( \frac{ad \overline{m\ell_1\ell_3 v}}{w} \right) = \widehat{G}(0)\mu(w)\frac{N}{w} \frac{\varphi(v)}{v},
\end{align*}
and upon summing this contribution over the remaining variables, the zero frequency contributes
\begin{align*}
\ll V^{\frac{1}{2}}W^{\frac{1}{2}} (MN)^{\frac{1}{2}} (y_1y_3)^{\frac{1}{2}} Q^{o(1)} \ll Q^{\frac{3}{2}}
\end{align*}
to $\mathcal{E}_1'$, and this contribution is sufficiently small.

It takes just a bit more work to bound the contribution of the nonzero frequencies $|h| > 0$. We rearrange the sum as
\begin{align*}
\sum_{d \mid v} \mu(d) \frac{N}{dw}\sum_{|h| \leq W^{1+\epsilon}d/N}\widehat{G} \left(\frac{hN}{dw} \right) \sum_{(a,w)=1} e \left(\frac{ad \overline{m\ell_1\ell_3 v}}{w} + \frac{ah}{w} \right).
\end{align*}
By a change of variables the inner sum is equal to the Ramanujan sum $c_w(hm\ell_1\ell_3v + d)$. Note that $hm\ell_1\ell_3v + d \neq 0$ because $m \ell_1\ell_3 \neq 1$. The nonzero frequencies therefore contribute to $\mathcal{E}_1'$ an amount
\begin{align*}
\ll Q^\epsilon \frac{(VWL_1L_3 M)^{\frac{1}{2}}}{N^{\frac{1}{2}}} \sup_{0 < |k| \ll Q^{O(1)}} \sum_{w \asymp W} |c_w (k)|.
\end{align*}
Since $|c_w(k)| \leq (k,w)$ the sum on $w$ is $\ll W^{1+o(1)}$. It follows that
\begin{align*}
\mathcal{E}_1' &\ll Q^{\frac{3}{2}}  + Q^{\frac{3}{2} + \epsilon} (y_1y_3)^{\frac{1}{2}} \frac{M^{1/2}}{N^{\frac{1}{2}}}.
\end{align*}
Since $y_i = Q^{\theta_i}$ and $\theta_i < \frac{1}{2} - 3\epsilon$, say, this bound for $\mathcal{E}_1'$ is acceptable provided $N \geq M Q^{-2\epsilon}$.
\end{proof}

By Proposition \ref{prop: cancellation in n} we may assume $N \leq M Q^{-2\epsilon}$. Since $MN \leq Q^{1 + \epsilon}$, the condition $N \leq M Q^{-2\epsilon}$ implies $N \leq Q^{\frac{1}{2}}$. 

We now pause to make a comment on the condition $m \ell_1 \ell_3 \neq 1$, which we have assumed throughout this section but not indicated in the notation for $\mathcal{E}_1'$. Observe that this condition is automatic if $M L_1 L_3 > 2018$. If $M L_1L_3 \ll 1$, then we may use the trivial bound \eqref{eq: E1' trivial bound} along with the bound $N \leq Q^{\frac{1}{2}} \leq  Q^{1-\epsilon}$ to obtain
\begin{align*}
\mathcal{E}_1' &\ll Q^{2-\epsilon}.
\end{align*}
We may therefore assume $ML_1L_2 \gg 1$, so that the condition $m \ell_1 \ell_3 \neq 1$ is satisfied.

We now remove the coprimality condition $(n,w) = 1$. By M\"obius inversion we have
\begin{align*}
\mathbf{1}_{(n,w)=1} = \sum_{\substack{f \mid n \\ f \mid w}} \mu(f).
\end{align*}
We move the sum on $f$ to be the outermost sum, and note $f \ll N$. We then change variables $n \rightarrow nf, w \rightarrow wf$. If $a_*$, say, is any lift of the multiplicative inverse of $m\ell_1\ell_3v$ modulo $wf$, then $a_* \equiv \overline{m\ell_1\ell_3v} \pmod{w}$, and therefore
\begin{align*}
\frac{nf \overline{m\ell_1\ell_3v}}{wf} \equiv \frac{n \overline{m\ell_1\ell_3v}}{w} \pmod{1}.
\end{align*}
It follows that
\begin{align*}
\mathcal{E}_1' &= \frac{W^{\frac{1}{2}}}{(MNL_1L_3V)^{\frac{1}{2}}}\sum_{f \ll N} \mu(f) \mathop{\sum \sum}_{(v,wf)=1} \alpha(v) G \left( \frac{v}{V}\right) G \left( \frac{wf}{W}\right) \\
&\times \mathop{\sum \sum}_{\substack{\ell_i \leq y_i \\ (\ell_i,fvw)=1}} \beta(\ell_1)\gamma(\ell_3) G \left( \frac{\ell_1}{L_1}\right) G \left( \frac{\ell_3}{L_3}\right) \mathop{\sum \sum}_{\substack{(m,fvw)=1 \\ (n,v)=1}} G \left( \frac{m}{M}\right) G \left( \frac{nf}{N}\right) e \left( \frac{n \overline{m\ell_1\ell_3v}}{w} \right).
\end{align*}

We next reduce the size of $f$ by a similar argument to the one that let us impose the condition $V \leq Q^\epsilon$. We obtain by transitioning to multiplicative characters (recall \eqref{eq: add char to mult char}) that the sum over $v,w,m,n,\ell_1,\ell_3$ is bounded by
\begin{align*}
\ll \frac{W^{\frac{1}{2}+o(1)} V^{\frac{1}{2}}}{f^{\frac{1}{2}}} \sum_{w \asymp W/f} \frac{1}{w^{\frac{1}{2}}} \left(\frac{(MN)^{\frac{1}{2}}}{f^{\frac{1}{2}}} + w^{\frac{1}{2}} \right)\left((L_1L_3)^{\frac{1}{2}} + w^{\frac{1}{2}} \right) \ll \frac{Q^{2+\epsilon}}{f^{\frac{3}{2}}},
\end{align*}
and therefore the contribution from $f > Q^{4\epsilon}$ is negligible.

Now the only barrier to applying Lemma \ref{lem: sum of kloos sum} is the conditions $(m,f) = 1$ and $(m,v) = 1$. We remove both of these conditions with M\"obius inversion, obtaining
\begin{align*}
&\sum_{f \ll \min(N,Q^\epsilon)} \mu(f) \sum_{h \mid f} \mu(h) \sum_{t \ll V} \mu(t) \frac{W^{\frac{1}{2}}}{(MNL_1L_3V)^{\frac{1}{2}}}\mathop{\sum \sum}_{\substack{(v,wf)=1 \\ (w,ht)=1}} \alpha(v) G \left( \frac{vt}{V}\right) G \left( \frac{wf}{W}\right) \\
&\times \mathop{\sum \sum}_{\substack{\ell_i \leq y_i \\ (\ell_i,fvw)=1}} \beta(\ell_1)\gamma(\ell_3) G \left( \frac{\ell_1}{L_1}\right) G \left( \frac{\ell_3}{L_3}\right) \mathop{\sum \sum}_{\substack{(m,w)=1 \\ (n,v)=1}} G \left( \frac{mht}{M}\right) G \left( \frac{nf}{N}\right) e \left( \frac{n \overline{mht^2\ell_1\ell_3v}}{w} \right).
\end{align*}
We set
\begin{align*}
b_{n,ht^2 k} = \mathbf{1}_{(n,v)=1}  G \left(\frac{nf}{N} \right)  \mathop{\sum_{\ell_1} \sum_{\ell_3} \sum_v}_{\substack{\ell_1 \ell_3 v = k \\ (\ell_1\ell_3,v)=1}} \beta(\ell_1) \gamma(\ell_3) \alpha(v) G \left( \frac{vt}{V}\right)G \left( \frac{\ell_1}{L_1}\right)G \left( \frac{\ell_3}{L_3}\right)
\end{align*}
if $(k,f)=1$, and for integers $r$ not divisible by $ht^2$ we set $b_{n,r}=0$. It follows that if $b_{n,r}\neq 0$, then $n \asymp N/f$ and $r \asymp ht L_1L_3V$ with $r \equiv 0 (ht^2)$. The sum over $n,r,m,w$ is therefore a sum of the form to which Lemma \ref{lem: sum of kloos sum} may be applied. We note that
\begin{align*}
\| b_{N,R}\|_2 \ll \frac{Q^{o(1)}}{(ft)^{\frac{1}{2}}} (NL_1L_3V)^{\frac{1}{2}},
\end{align*}
and therefore by Lemma \ref{lem: sum of kloos sum} we have
\begin{align*}
\mathcal{E}_1' &\ll Q^\epsilon \sum_{f \ll Q^\epsilon} \frac{1}{f^{\frac{1}{2}}} \sum_{h \mid f} \sum_{t \ll Q^\epsilon}\frac{1}{t^{\frac{1}{2}}} \frac{W^{\frac{1}{2}}}{M^{\frac{1}{2}}} \\
&\times\Bigg\{ \frac{W^{\frac{1}{2}}}{f^{\frac{1}{2}}} \left((htL_1L_3V)^{\frac{1}{2}} + \frac{N^{\frac{1}{2}}}{f^{\frac{1}{2}}} \right)\left(\frac{W^{\frac{1}{2}}}{f^{\frac{1}{2}}} + (ML_1L_3V)^{\frac{1}{2}} \right) \\
&+ \frac{W}{f} \frac{M^{\frac{1}{2}}}{(ht)^{\frac{1}{2}}}\left((htL_1L_3 V)^{\frac{1}{2}} + (ht L_1L_3 N V)^{\frac{1}{4}} \right) + \frac{M}{ht} (ht L_1L_3 N V)^{\frac{1}{2}} \Bigg\} \\
&\ll Q^{\epsilon} \Bigg(\frac{W^{\frac{3}{2}} (y_1y_3)^{\frac{1}{2}}}{M^{\frac{1}{2}}} + Wy_1y_3 + W^{\frac{3}{2}} \frac{N^{\frac{1}{2}}}{M^{\frac{1}{2}}} + W(y_1y_3)^{\frac{1}{2}} N^{\frac{1}{2}} + W^{\frac{3}{2}}(y_1y_3)^{\frac{1}{2}} \\ 
&+ W^{\frac{3}{2}}(y_1y_3)^{\frac{1}{4}}N^{\frac{1}{4}} + W^{\frac{1}{2}}Q^{\frac{1}{2}}(y_1y_3)^{\frac{1}{2}} \Bigg) \ll Q^{2 - \epsilon},
\end{align*}
upon recalling the bounds $W \ll Q$, $y_i \leq Q^{\theta_i}$ with $\theta_i < \frac{1}{2}$, and $N \leq Q^{1-\epsilon}$. This completes the proof of Lemma \ref{lem: main term IS MV}.

\section{Lemma \ref{lem: main term B MV}: main term}\label{sec: MV B main term}

In this section we obtain the main term of Lemma \ref{lem: main term B MV}. We allow ourselves to recycle some notation from sections \ref{sec: IS MV main term} and \ref{sec: IS MV error}.

Recall that we wish to asymptotically evaluate
\begin{align*}
\sideset{}{^+}\sum_{\chi(q)} \left| L \left( \frac{1}{2},\chi \right)\right|^2 \psi_{\text{B}}(\chi) \psi_{\text{MV}}(\overline{\chi}).
\end{align*}

We begin precisely as in section \ref{sec: IS MV main term}. Inserting the definitions of $\psi_\text{B}(\chi)$ and $\psi_{\text{MV}}(\overline{\chi})$, we must asympotically evaluate
\begin{align}\label{eq: full B MV expanded}
\frac{2}{L}\mathop{\sum \sum}_{\substack{bc \leq y_2 \\ (bc,q)=1}}\frac{\Lambda(b) \mu(c) P_2[bc]}{(bc)^{\frac{1}{2}}} \sum_{\substack{\ell \leq y_3 \\ (\ell,q)=1}} \frac{\mu(\ell) P_3[\ell]}{\ell^{\frac{1}{2}}} \mathop{\sum \sum}_{(mn,q)=1} \frac{1}{(mn)^{\frac{1}{2}}} V \left( \frac{mn}{q}\right) \sideset{}{^+}\sum_{\chi(q)} \epsilon(\chi) \chi(c \ell m) \overline{\chi}(bn).
\end{align}

The main term of Lemma \ref{lem: main term IS MV} arose from $m \ell_1 \ell_3 = 1$. In the present case, the main term contains more than just $c \ell m = 1$; the main term arises from those $c \ell m$ which divide $b$. The support of the von Mangoldt function constrains $b$ to be a prime power, so the condition $c \ell m \mid b$ is straightforward, but tedious, to handle. 

There are three different cases to consider. The first case is $c \ell m = 1$. In the second case we have $c \ell m = p$ and $b = p$. Both of these cases contribute to the main term. The third case is everything else ($b = p^j$ with $j \geq 2$ and $c\ell m \mid b$ with $c \ell m \geq p$), and this case contributes only to the error term.

\subsection{First case: $c \ell m = 1$} \ \\

If $c \ell m$ is equal to 1 then certainly $c \ell m$ divides $b$ for every $b$. The contribution from $c \ell m = 1$ is equal to
\begin{align*}
M &= \frac{2P_3(1)}{L} \sideset{}{^+}\sum_{\chi (q)}\epsilon(\chi) \sum_{b \leq y_2} \frac{\Lambda(b) \overline{\chi}(b) P_2[b]}{b^{\frac{1}{2}}} \sum_n \frac{\overline{\chi}(n)}{n^{\frac{1}{2}}} V \left( \frac{n}{q}\right).
\end{align*}
By an application of Proposition \ref{prop: change V for F},
\begin{align}\label{eq: MV B main term split}
M &= M_1 + M_2,
\end{align}
where
\begin{align*}
M_1 &= \frac{2P_3(1)}{L} \sum_{\substack{b \leq y_2 \\ (b,q)=1}} \frac{\Lambda(b) P_2[b]}{b^{\frac{1}{2}}} \sideset{}{^+}\sum_{\chi (q)} \epsilon(\chi) \overline{\chi}(b) L \left( \frac{1}{2},\overline{\chi} \right), \\
M_2 &= -\frac{2P_3(1)}{L} \mathop{\sum \sum}_{\substack{b \leq y_2 \\ (bn,q)=1}} \frac{\Lambda(b) P_2[b]}{(bn)^{\frac{1}{2}}} F(n) \sideset{}{^+}\sum_{\chi (q)} \chi(n) \overline{\chi}(b),
\end{align*}
and $F$ is the rapidly decaying function given by \eqref{eq: defn of func F}. A main term arises from $M_1$, and $M_2$ contributes only to the error term.

Let us first investigate $M_2$. By Lemma \ref{lem: character orthogonality} we have
\begin{align*}
M_2 &= -\frac{P_3(1)}{L}\sum_{w \mid q} \varphi(w) \mu(q/w) \mathop{\sum \sum}_{\substack{b \leq y_2 \\ b \equiv \pm n (w) \\ (bn,q)=1}} \frac{\Lambda(b) P_2[b]}{(bn)^{\frac{1}{2}}} F(n).
\end{align*}
By the rapid decay of $F$ (recall \eqref{eq: F has rapid decay}) we may restrict $n$ to $n \leq q^{\frac{1}{10}}$. The contribution from $w \leq q^{\frac{1}{2}+\epsilon}$ is then trivially $\ll q^{1 - \epsilon}$, since $y_2 \ll q^{\frac{1}{2}-\epsilon}$. For the remaining terms, the congruence condition $b \equiv \pm n (w)$ becomes $b = n$, and thus
\begin{align*}
M_2 &\ll q^{1-\epsilon} +\frac{1}{L} \sum_{\substack{w \mid q \\ w > q^{\frac{1}{2}+\epsilon}}} \varphi(w) \sum_{\substack{b \leq q^{\frac{1}{10}}}} \frac{\Lambda(b) P_2[b]}{b} F(b) \ll qL^{-1}.
\end{align*}

Let us turn to $M_1$. We use the following lemma to represent the central value $L \left( \frac{1}{2},\overline{\chi}\right)$.
\begin{lemma}\label{lem: first moment central value}
Let $\overline{\chi}$ be a primitive even character modulo $q$. Then
\begin{align*}
L \left( \frac{1}{2},\overline{\chi}\right) &= \sum_n \frac{\overline{\chi}(n)}{n^{\frac{1}{2}}} V_1 \left( \frac{n}{q^{\frac{1}{2}}}\right) + \epsilon(\overline{\chi}) \sum_n \frac{\chi(n)}{n^{\frac{1}{2}}} V_1 \left( \frac{n}{q^{\frac{1}{2}}}\right),
\end{align*}
where
\begin{align*}
V_1(x) &= \frac{1}{2\pi i}\int_{(1)} \frac{\Gamma \left( \frac{s}{2} + \frac{1}{4}\right)}{\Gamma \left( \frac{1}{4}\right)} \frac{G_1(s)}{s} \pi^{-s/2} x^{-s}ds
\end{align*}
and $G_1(s)$ is an even polynomial satisfying $G_1(0) = 1$.
\end{lemma}
\begin{proof}
See \cite[(2.2)]{IwaSar}.
\end{proof}

Applying Lemma \ref{lem: first moment central value}, the main term $M_1$ naturally splits as $M_1 = M_{1,1} + M_{1,2}$, where
\begin{align*}
M_{1,1} &= \frac{2P_3(1)}{L} \mathop{\sum \sum}_{\substack{b \leq y_2 \\ (bn,q)=1}} \frac{\Lambda(b) P_2[b]}{(bn)^{\frac{1}{2}}} V_1 \left( \frac{n}{q^{\frac{1}{2}}}\right)\sideset{}{^+}\sum_{\chi(q)} \epsilon(\chi) \overline{\chi}(bn), \\
M_{1,2} &= \frac{2P_3(1)}{L} \mathop{\sum \sum}_{\substack{b \leq y_2 \\ (bn,q)=1}} \frac{\Lambda(b) P_2[b]}{(bn)^{\frac{1}{2}}} V_1 \left( \frac{n}{q^{\frac{1}{2}}}\right)\sideset{}{^+}\sum_{\chi(q)} \chi(n)\overline{\chi}(b).
\end{align*}
Applying character orthogonality to $M_{1,1}$ we arrive at
\begin{align*}
M_{1,1} &= \frac{2P_3(1)}{L q^{1/2}} \mathop{\sum \sum}_{\substack{vw = q \\ (v,w)=1}} \mu^2(v) \varphi(w) \mathop{\sum \sum}_{\substack{b \leq y_2 \\ (bn,q)=1}} \frac{\Lambda(b) P_2[b]}{(bn)^{\frac{1}{2}}} V_1 \left( \frac{n}{q^{\frac{1}{2}}}\right)\cos\left(\frac{2\pi bn\overline{v}}{w} \right),
\end{align*}
and a trivial estimation shows
\begin{align*}
M_{1,1} \ll q^{1-\epsilon}.
\end{align*}

Let us lastly examine $M_{1,2}$, from which a main term arises. By character orthogonality we have
\begin{align*}
M_{1,2} &= \frac{P_3(1)}{L} \sum_{w \mid q} \varphi(w) \mu(q/w)\mathop{\sum \sum}_{\substack{b \leq y_2 \\ b \equiv \pm n (w) \\ (b,q)=1}} \frac{\Lambda(b) P_2[b]}{(bn)^{\frac{1}{2}}} V_1 \left( \frac{n}{q^{\frac{1}{2}}}\right).
\end{align*}
By trivial estimation, the contribution from $w \leq q^{\frac{1}{2} + \epsilon}$ is
\begin{align*}
\ll \sum_{\substack{w \mid q \\ w \leq q^{\frac{1}{2} + \epsilon}}} \varphi(w) \sum_{b \leq y_2} \frac{1}{b^{1/2}} \sum_{\substack{n \leq q^{\frac{1}{2} + \epsilon} \\ n \equiv \pm b (w)}} \frac{1}{n^{\frac{1}{2}}} \ll y_2^{\frac{1}{2}} \sum_{\substack{w \mid q \\ w \leq q^{\frac{1}{2} + \epsilon}}} \varphi(w) \left( \frac{q^{\frac{1}{4} + \epsilon}}{w} + O(1)\right) \ll q^{\frac{3}{4} + \epsilon}.
\end{align*}
By the rapid decay of $V_1$, for $w > q^{\frac{1}{2} + \epsilon}$ the congruence $b \equiv \pm n (w)$ becomes $b = n$. Adding back in the terms $w \leq q^{\frac{1}{2} +\epsilon}$, we have
\begin{align*}
M_{1,2} &= \frac{2P_3(1)}{L} \varphi^+(q) \sum_{\substack{b \leq y_2 \\ (b,q)=1}} \frac{\Lambda(b) P_2[b]}{b} V_1 \left( \frac{b}{q^{\frac{1}{2}}}\right) + O(q^{1-\epsilon}).
\end{align*}
For $x \ll 1$ we see by a contour shift that
\begin{align*}
V_1(x) &= 1 + O(x^{\frac{1}{3}}),
\end{align*}
and we have $b q^{-1/2} \ll q^{-\epsilon}$. It follows that
\begin{align*}
M_{1,2} &= O(q^{1-\epsilon}) + \frac{2P_3(1)}{L} \varphi^+(q) \sum_{\substack{b \leq y_2 \\ (b,q)=1}} \frac{\Lambda(b) P_2[b]}{b}.
\end{align*}
We have
\begin{align*}
\sum_{(b,q)>1} \frac{\Lambda(b)}{b} &\ll 1 + \sum_{p \mid q} \frac{\log p}{p} \ll \log \log q,
\end{align*}
and therefore we may remove the condition $(b,q) = 1$ at the cost of an error $O \left( q L^{-1+\epsilon}\right)$. From the estimate
\begin{align*}
\sum_{n \leq x} \frac{\Lambda(n)}{n} = \log x + O(1),
\end{align*}
summation by parts, and elementary manipulations, we obtain
\begin{align*}
\sum_{\substack{b \leq y_2}} \frac{\Lambda(b) P_2[b]}{b} &= (\log y_2) \int_0^1 P_2(u) du + O(1).
\end{align*}
Therefore, the contribution to the main term of Lemma \ref{lem: main term B MV} from $c \ell m = 1$ is
\begin{align}\label{eq: contrib from clm = 1}
(2\theta_2 P_3(1)\widetilde{P_2}(1) + o(1)) \varphi^+(q).
\end{align}

\subsection*{Second case: $c\ell m = p$, $b = p$} \ \\

Another main term which contributes to Lemma \ref{lem: main term B MV} comes from $c\ell m = p$ and $b = p$. There are three subcases: $(c,\ell,m) = (p,1,1), (1,p,1),$ or $(1,1,p)$. These three cases give (compare with \eqref{eq: full B MV expanded})
\begin{align*}
N_1 &= -\frac{2P_3(1)}{L} \sum_{\substack{p \leq y_2^{1/2} \\ (p,q)=1}} \frac{(\log p)P_2\left(\frac{\log(y_2^{1/2}/p)}{\log (y_2^{1/2})} \right)}{p} \sideset{}{^+}\sum_{\chi (q)} \epsilon(\chi) \sum_{n} \frac{\overline{\chi}(n)}{n^{\frac{1}{2}}} V \left( \frac{n}{q}\right), \\
N_2 &= -\frac{2}{L}\sum_{\substack{p \leq y_2 \\ (p,q)=1}} \frac{(\log p)P_2[p]P_3[p]}{p} \sideset{}{^+}\sum_{\chi (q)} \epsilon(\chi) \sum_{n} \frac{\overline{\chi}(n)}{n^{\frac{1}{2}}} V \left( \frac{n}{q}\right), \\
N_3 &= \frac{2P_3(1)}{L}  \sum_{\substack{p \leq y_2 \\ (p,q)=1}} \frac{(\log p)P_2[p]}{p}\sideset{}{^+}\sum_{\chi(q)} \epsilon(\chi) \sum_n \frac{\overline{\chi}(n)}{n^{\frac{1}{2}}} V \left( \frac{pn}{q}\right).
\end{align*}
The first two are somewhat easier to handle than the last one. We apply Proposition \ref{prop: change V for F} then argue as in section \ref{sec: IS MV main term} and the $c \ell m = 1$ case to obtain
\begin{align*}
\sideset{}{^+}\sum_{\chi (q)} \epsilon(\chi) \sum_{n} \frac{\overline{\chi}(n)}{n^{1/2}} V \left( \frac{n}{q}\right) &= \frac{1}{2}\varphi^+(q) + O(q^{1-\epsilon}).
\end{align*}
It follows that
\begin{align}\label{eq: intermed B MV main terms}
N_1 &= -\left(\frac{\theta_2}{2} P_3(1)\widetilde{P_2}(1) + o(1) \right) \varphi^+(q), \\
N_2 &= - \left(\theta_2 \int_0^1 P_2(u)P_3(u)du + o(1) \right) \varphi^+(q). \nonumber
\end{align}
Combining \eqref{eq: contrib from clm = 1} and \eqref{eq: intermed B MV main terms} gives the main term of Lemma \ref{lem: main term B MV}.

The final term $N_3$ is more difficult because the inner sum now depends on $p$. However, $M_3$ contributes only to the error term. By Proposition \ref{prop: change V for F} with $T = p$,
\begin{align}\label{eq: N3 decomp}
\sum_n \frac{\overline{\chi}(n)}{n^{\frac{1}{2}}} V \left( \frac{pn}{q}\right) &= L \left( \frac{1}{2},\overline{\chi}\right) -\sum_n \frac{\chi(n)}{n^{\frac{1}{2}}} F \left( \frac{n}{p}\right).
\end{align}
The first term on the right side of \eqref{eq: N3 decomp} contributes to $N_3$ an amount
\begin{align}\label{eq: N3 contrib part 1}
\left(2\theta_2 P_3(1)\widetilde{P_2}(1) + o(1)\right) \varphi^+(q).
\end{align}
For the second term on the right side of \eqref{eq: N3 decomp} we use character orthogonality and get
\begin{align*}
-\frac{2P_3(1)}{L}\sum_{\substack{p \leq y_2 \\ (p,q)=1}}\frac{(\log p)P_2[p]}{p}\frac{1}{2}\sum_{w \mid q} \varphi(w)\mu(q/w) \sum_{n \equiv \pm 1 (w)} \frac{1}{n^{\frac{1}{2}}} F \left( \frac{n}{p}\right).
\end{align*}
By the rapid decay of $F$ the contribution from $n > p^{\frac{11}{10}}$, say, is $O(qL^{-1})$. We next estimate trivially the contribution from $w \leq q^{\frac{3}{5}}$, say. We have the bound
\begin{align*}
\sum_{\substack{n \equiv \pm 1 (w) \\ n \leq p^{\frac{11}{10}}}} \frac{1}{n^{\frac{1}{2}}} F \left( \frac{n}{p}\right) \ll q^\epsilon \left(\frac{p^{\frac{11}{20}}}{w} + 1 \right),
\end{align*}
and this contributes to $N_3$ an amount
\begin{align*}
\ll q^{\frac{3}{5}+\epsilon} + q^\epsilon \sum_{p \leq y_2} p^{-\frac{9}{20}} \ll q^{\frac{3}{5}+\epsilon},
\end{align*}
since $y_2 \ll q^{\frac{1}{2}}$. For $w > q^{\frac{3}{5}}$ and $n \leq p^{\frac{11}{10}}$ the congruence $n \equiv \pm 1 (w)$ becomes $n = 1$. By a contour shift we have
\begin{align*}
F\left( \frac{1}{p}\right) &= 1 + O \left(p^{-\frac{1}{2}} \right).
\end{align*}
Thus, the second term on the right side of \eqref{eq: N3 decomp} contributes to $N_3$ an amount
\begin{align}\label{eq: N3 contrib second part}
-\left(2\theta_2 P_3(1)\widetilde{P_2}(1) + o(1)\right) \varphi^+(q),
\end{align}
and \eqref{eq: N3 contrib part 1} and \eqref{eq: N3 contrib second part} together imply $N_3$ is negligible.

\subsection*{Third case: everything else} \ \\

This case is the contribution from $b = p^j$ with $j \geq 2$ and $c\ell m \mid b$ with $c \ell m \geq p$. This case contributes an error of size $O(q L^{-1+\epsilon})$, essentially because the sum
\begin{align*}
\sum_{\substack{p^k \\ k \geq 2}} \frac{\log(p^k)}{p^k}
\end{align*}
converges. There are four different subcases to consider, since the M\"obius functions attached to $c$ and $\ell$ imply $c,\ell \in \{1,p\}$. The same techniques we have already employed allow one to bound the resulting sums, so we leave the details for the interested reader. This completes the proof of Lemma \ref{lem: main term B MV}.

\section{Lemma \ref{lem: main term B MV}: error term}\label{sec: B MV error}

After the results of the previous section, it remains to finish the proof of Lemma \ref{lem: main term B MV} by showing the error term of \eqref{eq: full B MV expanded} is negligible. The argument is very similar to that given in section \ref{sec: IS MV error}, and, indeed, the arguments are identical after a point.

The error term has the form
\begin{align*}
\mathcal{E}_2 &= \mathop{\sum \sum}_{(v,w)=1} \mu^2(v) \frac{v}{\varphi(v)} \frac{w^{\frac{1}{2}}}{v^{\frac{1}{2}}} \Psi\left(\frac{vw}{Q} \right)\mathop{\sum}_{\substack{\ell \leq y_3 \\ (\ell,vw)=1}} \frac{\mu(\ell)P_3[\ell]}{\ell^{\frac{1}{2}}} \\
&\times \mathop{\sum \sum}_{\substack{bc \leq y_2 \\ (bc,vw)=1}} \frac{\Lambda(b) \mu(c) P_2[bc]}{(bc)^{\frac{1}{2}}} \mathop{\sum \sum}_{(mn,vw)=1} \frac{1}{(mn)^{\frac{1}{2}}} V \left( \frac{mn}{vw}\right) \cos \left( \frac{2\pi bn \overline{c\ell m v}}{w} \right),
\end{align*}
where we also have the condition $c \ell m \nmid b$, which we do not indicate in the notation. This condition is awkward, but turns out to be harmless.

We note that we may separate the variables $b$ and $c$ from one another in $P_2[bc]$ by linearity, the additivity of the logarithm, and the binomial theorem. Thus, it suffices to study $\mathcal{E}_1$ with $P_2[bc]$ replaced by $(\log b)^{j_1} (\log c)^{j_2}$, for $j_i$ some fixed nonnegative integers. Arguing as in the reduction to \eqref{eq: desired bound for E1'}, we may bound $\mathcal{E}_2$ by $\ll Q^{o(1)}$ instances of $\mathcal{E}_2' = \mathcal{E}_2' (B,C,L,M,N,V,W)$, where
\begin{align}\label{eq: desired bound for E2'}
\mathcal{E}_2' &= \frac{W^{\frac{1}{2}}}{(BCLMNV)^{\frac{1}{2}}}\mathop{\sum \sum}_{(v,w)=1} \alpha(v) G \left( \frac{v}{V}\right) G \left( \frac{w}{W}\right) \nonumber \\
&\times \sum_{\substack{\ell \leq y_3 \\ (\ell,vw)=1}} \beta(\ell) G \left( \frac{\ell}{L}\right)\mathop{\sum \sum}_{\substack{bc \leq y_2 \\ (bc,vw)=1}} \gamma(b)\delta(c) G \left( \frac{b}{B}\right) G \left( \frac{c}{C}\right) \\
&\times \mathop{\sum \sum}_{(mn,vw)=1} G \left( \frac{m}{M}\right) G \left( \frac{n}{N}\right) e \left( \frac{bn \overline{c\ell m v}}{w} \right) \nonumber,
\end{align}
the function $G$ is smooth as before, and $\alpha,\beta,\gamma,\delta$ are sequences $f$ satisfying $|f(z)| \ll Q^{o(1)}$. We also have the conditions
\begin{align*}
VW \asymp Q, \ \ \ \ MN \leq Q^{1+\epsilon}, \ \ \ \ BC \ll y_2, \ \ \ \ L \ll y_3, \ \ \ \ B,C,L,M,N,V,W \gg 1.
\end{align*}
By the argument that gave \eqref{eq: E1' mult char orthog} we may also assume $V \leq Q^\epsilon$. Lastly, we may remove the condition $bc \leq y_2$ by Mellin inversion, at the cost of changing $\gamma$ and $\delta$ by $b^{it_0}, c^{it_0}$, respectively, where $t_0 \in \mathbb{R}$ is arbitrary (see \cite[Lemma 9]{DFI}, for instance).

Recall the condition $c\ell m\nmid b$. This condition is unnecessary if $CLM > 2018B$, so it is only in the case $CLM \ll B$ where we need to deal with it. However, the case $CLM \ll B$ is exceptional, since $B$ is bounded by $y_2 \ll Q^{\frac{1}{2}}$ but generically we would expect $CLM$ to be much larger than $Q^{\frac{1}{2}}$.

Indeed, we now show that when $CLM \ll B$ it suffices to get cancellation from the $n$ variable alone. The proof is essentially Proposition \ref{prop: cancellation in n}, so we just remark upon the differences. By M\"obius inversion and Poisson summation we have
\begin{align*}
\sum_{(n,vw)=1} G \left( \frac{n}{N} \right) e \left(\frac{bn \overline{c \ell m v}}{w} \right) &= \mu(w)\frac{N}{w}\frac{\varphi(v)}{v} \\ 
&+ \sum_{d \mid v} \mu(d) \frac{N}{dw}\sum_{|h| \leq W^{1+\epsilon}d/N}\widehat{G} \left(\frac{hN}{dw} \right) \sum_{(a,w)=1} e \left(\frac{abd \overline{c \ell m v}}{w} + \frac{ah}{w} \right) \\ 
&+ O(Q^{-100}).
\end{align*}
The first and third terms contribute acceptable amounts, so consider the second term. The sum over $a$ is the Ramanujan sum $c_w(hc\ell m v+bd)$, and since $c\ell m$ does not divide $b$ the argument of the Ramanujan sum is non-zero. Following the proof of Proposition \ref{prop: cancellation in n}, we therefore obtain a bound of
\begin{align}\label{eq: E2' after canc in n var}
\mathcal{E}_2' &\ll \frac{Q^{\frac{3}{2}+\epsilon} (BCLM)^{\frac{1}{2}}}{N^{\frac{1}{2}}}.
\end{align}
By the reasoning immediately after Proposition \ref{prop: cancellation in n}, the bound \eqref{eq: E2' after canc in n var} allows us to assume $N \leq M Q^{-2\epsilon}$, so that $N \leq Q^{\frac{1}{2}}$, regardless of whether $CLM \ll B$. In the case $CLM \ll B$, the bound \eqref{eq: E2' after canc in n var} becomes
\begin{align*}
\mathcal{E}_2' &\ll \frac{Q^{\frac{3}{2} + \epsilon} B}{N^{\frac{1}{2}}} \ll Q^{\frac{3}{2} + \epsilon} B \ll Q^{\frac{3}{2} + \theta_2 + \epsilon} \ll Q^{2-\epsilon},
\end{align*}
which of course is acceptable.

At this point we can follow the rest of the proof in section \ref{sec: IS MV error}. We change variables $bn \rightarrow n$, and the rest follows \emph{mutatis mutandis} (it is important that with $N \ll Q^{\frac{1}{2}}$ we have $BN \ll Q^{1-\epsilon}$). This completes the proof of Lemma \ref{lem: main term B MV}.

\section{Acknowledgments}

The author thanks Kevin Ford and George Shakan for helpful comments on earlier drafts of this work. The author also thanks the anonymous referee for helpful comments. The author was supported in this work by NSF grant DMS-1501982, and is grateful to Kevin Ford for financial support.

\end{document}